\documentclass[12pt, reqno]{amsart}
\usepackage{ amsmath,amsthm, amscd, amsfonts, amssymb, graphicx, color}
\usepackage[bookmarksnumbered, colorlinks, plainpages]{hyperref}
\newtheorem{thm}{Theorem}[section]
\newtheorem{cor}[thm]{Corollary}
\newtheorem{lem}[thm]{Lemma}

\newtheorem{exam}[thm]{Example}
\numberwithin{equation}{section}

\begin{document}

\title{Representation of the g-Drazin inverse in a Banach algebra}
\author{Huanyin Chen}
\address{
Department of Mathematics\\ Hangzhou Normal University\\ Hang -zhou, China}
\email{<huanyinchen@aliyun.com>}
\author{Marjan Sheibani Abdolyousefi}
\address{Women's University of Semnan (Farzanegan), Semnan, Iran}
\email{<sheibani@fgusem.ac.ir>}


\begin{abstract}
Let $\mathcal{A}$ be a complex Banach algebra. An element $a\in \mathcal{A}$ has g-Drazin inverse if there exists $b\in \mathcal{A}$ such that $$b=bab, ab=ba, a-a^2b\in \mathcal{A}^{qnil}.$$ Let $a,b\in \mathcal{A}$ have g-Drazin inverses. If $$ab=\lambda a^{\pi}b^{\pi}bab^{\pi},$$ we prove that
$a+b$ has g-Drazin inverse
and $$(a+b)^d=b^{\pi}a^d+b^da^{\pi}+\sum\limits_{n=0}^{\infty}(b^d)^{n+1}a^na^{\pi}+\sum\limits_{n=0}^{\infty}b^{\pi}(a+b)^nb(a^d)^{n+2}.$$
The main results of Mosic (Bull. Malays. Sci. Soc., {\bf 40}(2017), 1465--1478) is thereby extended to the general case.
Applications to block operator matrices are given.\end{abstract}

\maketitle

\section{Introduction}

Throughout the paper, $\mathcal{A}$ is a complex Banach algebra with an identity and $\lambda$ always stands for a nonzero complex number. The commutant of $a\in \mathcal{A}$ is defined by $comm(a)=\{x\in
\mathcal{A}~|~xa=ax\}$. An element $a$ in $\mathcal{A}$ has g-Drazin inverse provided that there exists $b\in comm(a)$ such that $b=bab$
and $a-a^2b\in \mathcal{A}^{qnil}.$ Here, $\mathcal{A}^{qnil}$ is the set of all quasinilpotents in $\mathcal{A}$, i.e.,
$$\mathcal{A}^{qnil}=\{a\in \mathcal{A}~|~1+ax\in U(\mathcal{A})~\mbox{for
every}~x\in comm(a)\}.$$  For a Banach algebra $\mathcal{A}$ it is well known
 that $$a\in \mathcal{A}^{qnil}\Leftrightarrow
\lim\limits_{n\to\infty}\parallel a^n\parallel^{\frac{1}{n}}=0\Leftrightarrow 1+\lambda a\in \mathcal{A}^{-1}~\mbox{for any}~ \lambda\in {\Bbb C}.$$
Here, $\mathcal{A}^{-1}$ stands for the set of all invertible elements in $\mathcal{A}$. We always use $\mathcal{A}^{d}$ to denote the set of all g-Drazin invertible elements in $\mathcal{A}$. It was proved that $a\in \mathcal{A}^{d}$ if and only if there exists an idempotent $p\in comm(a)$ such that $a+p\in \mathcal{A}^{-1}$ and $ap\in \mathcal{A}^{qnil}$ (see ~\cite[Theorem 4.2]{K}).

The g-Drazin invertiblity of the sum of two  elements has a rich history. It was extensively studied in matrix and operator theory by many authors, e.g., \cite{B, CM, CL, DW, DW2, M1, M, MZ2} and \cite{ZM}. In this paper we study this topic under the weaker conditions and extend the main results of \cite{M}.

Let $a,b\in \mathcal{A}^{d}$ and $$ab=\lambda a^{\pi}b^{\pi}bab^{\pi}.$$ In Section 2, we prove that $a+b$ has generalized Drazin inverse and the explicit formula for $(a+b)^d$ is obtained. This extend the result of \cite[Theorem 2.2]{M}.

Let $x\in \mathcal{A}$, and let $p^2=p\in \mathcal{A}$. Then we have Pierce matrix decomposition $x=pxp+px(1-p)+(1-p)xp+(1-p)x(1-p)$. Set $a=pxp, b=px(1-p), c=(1-p)xp, d=(1-p)x(1-p)$. We use the following matrix version to express the Pierce matrix decomposition of $x$ about the idempotent $p$:
$$x=\left(
\begin{array}{cc}
a&b\\
c&d
\end{array}
\right)_p.$$
In Section 3, we characterize the generalized Drazin inverse of an element based on its pierce decomposition. These extend the result of \cite[Theorem 3.1]{M} to more general setting.

If $a\in \mathcal{A}$ has g-Drazin inverse $a^{d}$. The element $a^{\pi}=1-aa^{d}$ is called the spectral idempotent of $a$.
In the last section we then apply our main results to characterize the generalized Drazin inverse of block operator matrices under various spectral idempotent conditions.

\section{Additive results}

The purpose of this section is to investigate the g-Drazin invertibility of the sum of two elements in a Banach algebra. We start by

\begin{lem} Let $$x=\left(\begin{array}{cc}
a&0\\
c&b
\end{array}
\right)_p~\mbox{or}~\left(\begin{array}{cc}
b&c\\
0&a
\end{array}
\right)_p$$  Then $$x^d=\left(\begin{array}{cc}
a^d&0\\
z&b^d
\end{array}
\right)_p,~\mbox{or}~ \left(\begin{array}{cc}
b^d&z\\
0&a^d
\end{array}
\right)_p,$$ where $$\begin{array}{c}
z=(b^d)^2\big(\sum\limits_{i=0}^{\infty}(b^d)^ica^i\big)a^{\pi}+b^{\pi}\big(\sum\limits_{i=0}^{\infty}b^ic(a^d)^i\big)(a^d)^2-b^dca^d.
\end{array}$$
\end{lem}
\begin{proof} See ~\cite[Theorem 2.1]{C}.\end{proof}

\begin{lem} Let $\mathcal{A}$ be a Banach algebra, and let $a,b\in\mathcal{A}^{qnil}$. If $ab=\lambda ba$, then $a+b\in \mathcal{A}^{qnil}$.
\end{lem}
\begin{proof} See~\cite[Lemma 2.1]{C1} and ~\cite[Lemma 2.1]{DW2}.\end{proof}

\begin{lem} Let $\mathcal{A}$ be a Banach algebra, and let $a\in \mathcal{A}^{qnil},b\in \mathcal{A}^{d}$. If $$ab=\lambda b^{\pi}bab^{\pi},$$ then $a+b\in \mathcal{A}^{d}$ and
$$(a+b)^d=\sum\limits_{n=0}^{\infty}(b^d)^{n+1}a^n.$$\end{lem}
\begin{proof} Let $p=bb^{d}$. Then we have
$$b=\left(\begin{array}{cc}
b_1&0\\
0&b_2
\end{array}
\right)_p, a=\left(\begin{array}{cc}
a_{1}&a_{2}\\
a_{3}&a_4
\end{array}
\right)_p.$$ Moreover,
$$b^d=\left(\begin{array}{cc}
b_1^{-1}&0\\
0&0
\end{array}
\right)_p~\mbox{and}~b^{\pi}=\left(\begin{array}{cc}
0&0\\
0&1-bb^d
\end{array}
\right)_p.$$ By hypothesis, we have
$$\left(\begin{array}{cc}
a_1b_1&a_2b_2\\
a_3b_1&a_4b_2
\end{array}
\right)_p=ab=\lambda b^{\pi}bab^{\pi}=\left(\begin{array}{cc}
0&0\\
0&\lambda b_2a_4
\end{array}
\right)_p.$$
Then we get $a_1b_1=0$ and $a_3b_1=0$. It follows that $a_1=0$ and $a_3=0$. We easily see that $b_2=b-b^2b^d\in ((1-p)\mathcal{A}(1-p))^{qnil}$.
By hypothesis, $abb^d=0$, and so $a(1-bb^d)=a\in \mathcal{A}^{qnil}$. By using Cline's formula, we see that $(1-bb^d)a(1-bb^d)\in \mathcal{A}^{qnil}$. Hence, $a_4\in ((1-p)\mathcal{A}(1-p))^{qnil}$. Since $a_4b_2=\lambda b_2a_4$, it follows by Lemma 2.2 that
$a_4+b_2\in ((1-p)\mathcal{A}(1-p))^{qnil}$, and so $(a_4+b_2)^d=0$.

Then $$a+b=\left(\begin{array}{cc}
b_1&a_2\\
0&a_4+b_2
\end{array}
\right)_p.$$ In light of Lemma 2.1, we have $$(a+b)^d=\left(\begin{array}{cc}
b_1&a_2\\
0&a_4+b_2
\end{array}
\right)^d=\left(\begin{array}{cc}
b_1^{-1}&z\\
0&0
\end{array}
\right)_p,$$ where $z=\sum\limits_{n=0}^{\infty}b_1^{-(n+2)}a_2a_4^n=\sum\limits_{n=0}^{\infty}(b^d)^{n+2}a^{n+1}b^{\pi}.$ Therefore $$(a+b)^d=b^d+\sum\limits_{n=0}^{\infty}(b^d)^{n+2}a^{n+1}=\sum\limits_{n=0}^{\infty}(b^d)^{n+1}a^n,$$ as asserted.\end{proof}

We are now ready to prove the following.

\begin{thm} Let $\mathcal{A}$ be a Banach algebra, and let $a,b\in \mathcal{A}^{d}$. If $$ab=\lambda a^{\pi}b^{\pi}bab^{\pi},$$ then $a+b\in \mathcal{A}^{d}$ and
$$(a+b)^d=b^{\pi}a^d+b^da^{\pi}+\sum\limits_{n=1}^{\infty}(b^d)^{n+1}a^na^{\pi}+\sum\limits_{n=0}^{\infty}b^{\pi}(a+b)^nb(a^d)^{n+2}.$$\end{thm}
\begin{proof} Let $p=aa^{d}$. Then we have
$$a=\left(\begin{array}{cc}
a_1&0\\
0&a_2
\end{array}
\right)_p, b=\left(\begin{array}{cc}
b_{11}&b_{12}\\
b_{1}&b_2
\end{array}
\right)_p.$$ Since $ab=\lambda a^{\pi}b^{\pi}bab^{\pi},$ we see that
$a^db=(a^d)^2ab=\lambda (a^d)^2a^{\pi}b^{\pi}bab^{\pi}=0$; hence, $b_{11}=b_{12}=0$. Thus we get
$$a=\left(\begin{array}{cc}
a_1&0\\
0&a_2
\end{array}
\right)_p, b=\left(\begin{array}{cc}
0&0\\
b_{1}&b_2
\end{array}
\right)_p.$$ Moreover, $a_2=(1-p)a(1-p)=a-a^2a^{d}\in \mathcal{A}^{qnil}$.
In light of Cline's formula,
$b_2=a^{\pi}ba^{\pi}\in \mathcal{A}^{d}$, and so $b_2\in ((1-p)\mathcal{A}(1-p))^{d}$.
By hypothesis, we have
$$\begin{array}{ll}
\left(
\begin{array}{cc}
0&0\\
a_2b_1&a_2b_2
\end{array}
\right)&=ab=\lambda a^{\pi}b^{\pi}bab^{\pi}\\
&=\lambda \left(
\begin{array}{cc}
0&0\\
b_2^{\pi}b_1a_1-b_2^{\pi}b_2a_2b_2^db_1&b_2^{\pi}b_2a_2b_2^{\pi}
\end{array}
\right),
\end{array}$$ and so $$a_2b_2=\lambda b_2^{\pi}b_2a_2b_2^{\pi}.$$ In view of Lemma 2.3,
$$(a_2+b_2)^d=\sum\limits_{n=0}^{\infty}(b_2^d)^{n+1}a_2^n.$$
By virtue of Lemma 2.1, we have $$(a+b)^d=\left(\begin{array}{cc}
a_1^d&0\\
z&(a_2+b_2)^d
\end{array}
\right)=\left(\begin{array}{cc}
a^d&0\\
z&(a_2+b_2)^d\end{array}
\right),$$ where $$z=\sum\limits_{i=0}^{\infty}(a_2+b_2)^{\pi}(a_2+b_2)^ib_1(a^d)^{i+2}-(a_2+b_2)^db_1a^d.$$
Moreover, we have $a_2b_2^d=(a_2b_2)(b_2^d)^2=(\lambda b_2^{\pi}b_2a_2b_2^{\pi})(b_2^d)^2=0$, and then
$$\begin{array}{lll}
(a_2+b_2)^{\pi}&=&(1-p)-b_2\sum\limits_{n=0}^{\infty}(b_2^d)^{n+1}a_2^n\\
&=&b_2^{\pi}-\sum\limits_{n=0}^{\infty}(b_2^d)^{n+1}a_2^{n+1}.
\end{array}$$

Since $a_2b_1=\lambda (b_2^{\pi}b_1a_1-b_2^{\pi}b_2a_2b_2^db_1)=\lambda b_2^{\pi}b_1a_1$, we see that
$b_2^da_2b_1=b_2^d(\lambda b_2^{\pi}b_1a_1)=0$, and so
$$b_2^da_2^2b_1=(b_2^da_2)=(b_2^da_2)(\lambda b_2^{\pi}b_1a_1)=\lambda (b_2^da_2)\big(1-b_2b_2^{d})b_1a_1\big)=0.$$
By induction, we have $b_2^da_2^nb_1=0 (n\geq 1)$. Likewise, we have $b_2^da_2^nb_2=0 (n\geq 1)$. Therefore
$$(b_2^d)^{k+1}a_2^{k+1}(a_2+b_2)^nb_1=0 (k, n\geq 0).$$ This shows that $$\begin{array}{lll}
z&=&\sum\limits_{i=0}^{\infty}\big(b_2^{\pi}-\sum\limits_{k=0}^{\infty}(b_2^d)^{k+1}a_2^{k+1}\big)(a_2+b_2)^ib_1(a^d)^{i+2}-(a_2+b_2)^db_1a^d\\
&=&\sum\limits_{i=0}^{\infty}b_2^{\pi}(a_2+b_2)^ib_1(a^d)^{i+2}-\sum\limits_{i=0}^{\infty}\sum\limits_{k=0}^{\infty}(b_2^d)^{k+1}a_2^{k+1}(a_2+b_2)^ib_1(a^d)^{i+2}\\
&-&b_2^db_1a^d-\sum\limits_{n=1}^{\infty}(b_2^d)^{n+1}a_2^nb_1a^d\\
&=&\sum\limits_{i=0}^{\infty}b_2^{\pi}(a_2+b_2)^ib_1(a^d)^{i+2}-b_2^db_1a^d.
\end{array}$$
Since $a^db=0$, we check that
$$\begin{array}{c}
\left(\begin{array}{cc}
a^d&0\\
-b_2^db_1a^d&0
\end{array}
\right)=b^{\pi}a^d,\\
\left(\begin{array}{cc}
0&0\\
b_2^{\pi}(a_2+b_2)^{i}b_1(a^d)^{i+2}&0
\end{array}
\right)=b^da^{\pi}+\sum\limits_{i=0}^{\infty}(b^d)^{i+1}a^ia^{\pi}.
\end{array}$$
Moreover, we have
$$\left(\begin{array}{cc}
0&0\\
0&(b_2^d)^{i+1}a_2^i
\end{array}
\right)=(b^d)^{i+1}a^ia^{\pi}.$$
Therefore $$(a+b)^d=b^{\pi}a^d+b^da^{\pi}+\sum\limits_{n=1}^{\infty}(b^d)^{n+1}a^na^{\pi}+\sum\limits_{n=0}^{\infty}b^{\pi}(a+b)^nb(a^d)^{n+2},$$ as asserted.
\end{proof}

\section{Pierce matrix decomposition}

Every element in a Banach algebra can be represented by a Pierce form. It is attractive to investigate the g-Drazin inverse of an element is a Pierce form.
Let $x\in \mathcal{A}$ given by the Pierce form $x=\left(
\begin{array}{cc}
a&b\\
c&d
\end{array}
\right)_p$ for an idempotent $p\in \mathcal{A}$. We now extend the result of \cite[Theorem 3.1]{M} to the following.

\begin{thm} If $$a^{\pi}bc=0, a^{\pi}bd=\lambda ab~\mbox{and}~cb+\lambda^2\sum\limits_{n=1}^{\infty}(d^d)^nca^nb=0,$$ then $x\in \mathcal{A}^d$ and
$$\begin{array}{lll}
x^d&=&\left(
\begin{array}{cc}
a^d&0\\
u&d^d
\end{array}
\right)\\
&+&\sum\limits_{n=0}^{\infty}x^n\left(
\begin{array}{cc}
i_n-\sum\limits_{k=1}^{n}b(d^d)^{k+1}c(a^d)^{n+2-k}&b(d^d)^{n+2}\\
0&0
\end{array}
\right),
\end{array}$$ where
$$u=\sum\limits_{n=0}^{\infty}(d^d)^{n+2}ca^na^{\pi}+\sum\limits_{n=0}^{\infty}d^{\pi}d^nc(a^d)^{n+2}-d^dca^d,$$
$$\begin{array}{lll}
i_n&=&\sum\limits_{k=0}^{\infty}bd^{\pi}d^kc(a^d)^{n+k+3}-bd^dc(a^d)^{n+2}\\
&+&\sum\limits_{k=0}^{\infty}b(d^d)^{n+k+3}ca^ka^{\pi}-b(d^d)^{n+2}ca^d
\end{array}$$ for $n\geq 0.$\end{thm}
\begin{proof} Clearly, we have $x=y+z$, where $$y=\left(
\begin{array}{cc}
a&0\\
c&d
\end{array}
\right), z=\left(
\begin{array}{cc}
0&b\\
0&0
\end{array}
\right).$$
Then $$y^d=\left(
\begin{array}{cc}
a^d&0\\
u&d^d
\end{array}
\right), z^d=\left(
\begin{array}{cc}
0&0\\
0&0
\end{array}
\right),$$ where $$u=\sum\limits_{n=0}^{\infty}(d^d)^{n+2}ca^na^{\pi}+\sum\limits_{n=0}^{\infty}d^{\pi}d^nc(a^d)^{n+2}-d^dca^d.$$
Hence, $$y^{\pi}=\left(
\begin{array}{cc}
a^{\pi}&0\\
-ca^d-du&d^{\pi}
\end{array}
\right), z^{\pi}=I_2.$$
This shows that $$yz=\left(
\begin{array}{cc}
0&ab\\
0&cb
\end{array}
\right), y^{\pi}z^{\pi}zyz^{\pi}=\left(
\begin{array}{cc}
a^{\pi}bc&a^{\pi}bd\\
(-ca^d-du)bc&(-ca^d-du)bd
\end{array}
\right).$$
Since $a^{\pi}bd=\lambda ab$, we have $a^db=(a^d)^2(ab)=\lambda^{-1}(a^d)^2a^{\pi}bd=0.$ Hence, $bc=a^{\pi}bc=0$ and $$cb+\lambda dubd
=cb+\lambda \sum\limits_{n=0}^{\infty}(d^d)^{n+1}ca^nbd=cb+\lambda^2\sum\limits_{n=0}^{\infty}(d^d)^{n+1}ca^{n+1}b=0.$$
Then we compute that $$\begin{array}{c}
a^{\pi}bc=0;\\
(-ca^d-du)bc=0;\\
\lambda a^{\pi}bd=ab;\\
\lambda (-ca^d-du)bd=cb.
\end{array}$$ Therefore we have $yz=\lambda y^{\pi}z^{\pi}zyz^{\pi}.$ In light of Theorem 2.4, we derive
$$\begin{array}{lll}
x^d&=&y^d+\sum\limits_{n=0}^{\infty}x^nz(y^d)^{n+2}\\
&=&\left(
\begin{array}{cc}
a^d&0\\
u&d^d
\end{array}
\right)+\sum\limits_{n=0}^{\infty}x^n\left(
\begin{array}{cc}
0&b\\
0&0
\end{array}
\right)\left(
\begin{array}{cc}
a^d&0\\
u&d^d
\end{array}
\right)^{n+2},
\end{array}$$ as required.\end{proof}

\begin{cor} If $$d^{\pi}cb=0, d^{\pi}ca=\lambda dc~\mbox{and}~bc+\lambda^2\sum\limits_{n=1}^{\infty}(a^d)^nbd^nc=0,$$ then $x\in \mathcal{A}^d$ and
$$\begin{array}{lll}
x^d&=&\left(
\begin{array}{cc}
a^d&u\\
0&d^d
\end{array}
\right)\\
&+&\sum\limits_{n=0}^{\infty}x^n\left(
\begin{array}{cc}
0&0\\
c(a^d)^{n+2}&i_n-\sum\limits_{k=1}^{n}c(a^d)^{k+1}b(d^d)^{n+2-k}
\end{array}
\right),
\end{array}$$ where $$u=\sum\limits_{n=0}^{\infty}(a^d)^{n+2}bd^nd^{\pi}+\sum\limits_{n=0}^{\infty}a^{\pi}a^nb(d^d)^{n+2}-a^dbd^d,$$
$$\begin{array}{lll}
i_n&=&\sum\limits_{k=0}^{\infty}ca^{\pi}a^kb(d^d)^{n+k+3}-ca^db(d^d)^{n+2}\\
&+&\sum\limits_{k=0}^{\infty}c(a^d)^{n+k+3}bd^kd^{\pi}-c(a^d)^{n+2}bd^d
\end{array}$$ for $n\geq 0.$\end{cor}
\begin{proof} We easily see that $$x=\left(
\begin{array}{cc}
a&b\\
c&d
\end{array}
\right)_p=\left(
\begin{array}{cc}
0&1\\
1&0
\end{array}
\right)\left(
\begin{array}{cc}
d&c\\
b&a
\end{array}
\right)_{1-p}\left(
\begin{array}{cc}
0&1\\
1&0
\end{array}
\right).$$ Applying Theorem 3.1 to $\left(
\begin{array}{cc}
d&c\\
b&a
\end{array}
\right)_{1-p}$, we see that it has g-Drazin inverse and $$\begin{array}{lll}
\left(
\begin{array}{cc}
d&c\\
b&a
\end{array}
\right)^d&=&\left(
\begin{array}{cc}
d^d&0\\
u&a^d
\end{array}
\right)\\
&+&\sum\limits_{n=0}^{\infty}x^n\left(
\begin{array}{cc}
i_n-\sum\limits_{k=1}^{n}c(a^d)^{k+1}b(d^d)^{n+2-k}&c(a^d)^{n+2}\\
0&0
\end{array}
\right),
\end{array}$$ where
$$u=\sum\limits_{n=0}^{\infty}(a^d)^{n+2}bd^nd^{\pi}+\sum\limits_{n=0}^{\infty}a^{\pi}a^nb(d^d)^{n+2}-a^dbd^d,$$
$$\begin{array}{lll}
i_n&=&\sum\limits_{k=0}^{\infty}ca^{\pi}a^kb(d^d)^{n+k+3}-ca^db(d^d)^{n+2}\\
&+&\sum\limits_{k=0}^{\infty}c(a^d)^{n+k+3}bd^kd^{\pi}-c(a^d)^{n+2}bd^d
\end{array}$$ for $n\geq 0.$ Therefore $$\begin{array}{lll}
x^d&=&\left(
\begin{array}{cc}
0&1\\
1&0
\end{array}
\right)\left(
\begin{array}{cc}
d&c\\
b&a
\end{array}
\right)^d\left(
\begin{array}{cc}
0&1\\
1&0
\end{array}
\right)\\
&=&\left(
\begin{array}{cc}
a^d&u\\
0&d^d
\end{array}
\right)\\
&+&\sum\limits_{n=0}^{\infty}x^n\left(
\begin{array}{cc}
0&0\\
c(a^d)^{n+2}&i_n-\sum\limits_{k=1}^{n}c(a^d)^{k+1}b(d^d)^{n+2-k}
\end{array}
\right),
\end{array}$$ as required.\end{proof}

We next generalize ~\cite[Theorem 3.2]{M} to the following.

\begin{thm} If $$a^{\pi}abd^{\pi}=\lambda bd, d^{\pi}cb=(ca^d+du)ab~\mbox{and}~\sum\limits_{n=0}^{\infty}bd^nc(a^d)^n=0,$$ then $x\in \mathcal{A}^d$ and
$$x^d=\left(
\begin{array}{cc}
a^d&(a^d)^2b\\
u&d^d+i
\end{array}
\right),$$ where
$$i=\sum\limits_{n=0}^{\infty}d^{\pi}d^nc(a^d)^{n+3}b-d^dc(a^d)^2b+\sum\limits_{n=0}^{\infty}(d^d)^{n+3}ca^na^{\pi}b-(d^d)^2ca^db.$$\end{thm}
\begin{proof} Write $x=y+z$ as in Theorem 3.1.
Then $$zy=\left(
\begin{array}{cc}
bc&bd\\
0&0
\end{array}
\right), z^{\pi}y^{\pi}yzy^{\pi}=\left(
\begin{array}{cc}
a^{\pi}ab(-ca^d-du)&a^{\pi}abd^{\pi}\\
t(-ca^d-du)&td^{\pi}
\end{array}
\right),$$ where $t=(-ca^d-du)ab+d^{\pi}cb$. Since $a^{\pi}abd^{\pi}=\lambda bd$, we have $bd^d=0$, and so $a^{\pi}ab=\lambda bd$. By hypothesis, $t=0$. Moreover, we have $$\begin{array}{lll}
\lambda^{-1} a^{\pi}ab(-ca^d-du)&=&-bdca^d-bd^2u\\
&=&-bdca^d-\sum\limits_{n=0}^{\infty}bd^{n+2}c(a^d)^{n+2}\\
&=&bc.
\end{array}$$
Hance, we compute that $$\begin{array}{c}
t(-ca^d-du)=0;\\
td^{\pi}=0;\\
\lambda^{-1} a^{\pi}ab(-ca^d-du)=bc;\\
\lambda^{-1} a^{\pi}abd^{\pi}=bd.
\end{array}$$ Therefore we have $zy=\lambda^{-1}z^{\pi}y^{\pi}yzy^{\pi}.$ In light of Theorem 2.4,, we derive
$$\begin{array}{lll}
x^d&=&y^d+(y^d)^2z\\
&=&\left(
\begin{array}{cc}
a^d&(a^d)^2b\\
u&d^d+ua^db+d^dub
\end{array}
\right),
\end{array}$$ as desired.\end{proof}

As an immediate consequence of Theorem 3.3, we now derive

\begin{cor} If $$d^{\pi}dca^{\pi}=\lambda ca, a^{\pi}bc=(bd^d+au)dc~\mbox{and}~\sum\limits_{n=0}^{\infty}ca^nb(d^d)^n=0,$$ then $x\in \mathcal{A}^d$ and
$$x^d=\left(
\begin{array}{cc}
a^d+i&u\\
(d^d)^2c&d^d
\end{array}
\right),$$ where
$$i=\sum\limits_{n=0}^{\infty}a^{\pi}a^nb(d^d)^{n+3}c-a^db(d^d)^2c+\sum\limits_{n=0}^{\infty}(a^d)^{n+3}bd^nd^{\pi}c-(a^d)^2bd^dc.$$\end{cor}
\begin{proof} As in the proof of Corollary 3.2, we complete the proof by applying Theorem 3.3, to $\left(
\begin{array}{cc}
d&c\\
b&a
\end{array}
\right)$.\end{proof}

\begin{exam} Let $\mathcal{A}=M_3({\Bbb C}), p=\left(
\begin{array}{cc}
I_3&{\bf 0}_{3\times 3}\\
{\bf 0}_{3\times 3}&{\bf 0}_{3\times 3}
\end{array}
\right)\in M_6({\Bbb C})$, and let $x=\left(
\begin{array}{cc}
a&b\\
c&d
\end{array}
\right)_p\in M_6({\Bbb C})$. Then $$a^{\pi}bc=0, a^{\pi}bd=\lambda ab~\mbox{and}~\sum\limits_{n=0}^{\infty}(d^d)^nca^nb=0,$$ while $$a^{\pi}bc=0, a^{\pi}bd\neq ab~\mbox{and}~\sum\limits_{n=0}^{\infty}(d^d)^nca^nb=0.$$ In this case,
$x^d=0.$\end{exam}
\begin{proof} Let $M=\left(
\begin{array}{cc}
A&B\\
C&D
\end{array}
\right)$, where $$A=D=\left(
\begin{array}{ccc}
0&1&0\\
0&0&1\\
0&0&0
\end{array}
\right),
B=\left(
\begin{array}{ccc}
0&2&0\\
0&0&1\\
0&0&0
\end{array}
\right)~\mbox{and}~C=\left(
\begin{array}{ccc}
0&0&1\\
0&0&0\\
0&0&0
\end{array}
\right).$$
Let $p=\left(
\begin{array}{cc}
I_3&{\bf 0}_{3\times 3}\\
{\bf 0}_{3\times 3}&{\bf 0}_{3\times 3}
\end{array}
\right), a=pMp, b=pM(I_3-p), c=(I_3-p)Mp$ and $d=(I_3-p)M(I_3-p)$. It is easy to see that
$a=\left(
\begin{array}{cc}
A&0\\
0&0
\end{array}
\right), b=\left(
\begin{array}{cc}
0&B\\
0&0
\end{array}
\right)$ and
$c=\left(
\begin{array}{cc}
0&0\\
C&0
\end{array}
\right), d= \left(
\begin{array}{cc}
0&0\\
0&D
\end{array}
\right)$. As $a,d$ are nilpotent, we have $a^d=d^d=0$. Also it is easy to see that $bc=cb=0$ and $bd=2ab$. Then $$a^{\pi}bc=0, a^{\pi}bd=2 ab~\mbox{and}~\sum\limits_{n=0}^{\infty}(d^d)^nca^nb=0,$$ while $$a^{\pi}bc=0, a^{\pi}bd\neq ab~\mbox{and}~\sum\limits_{n=0}^{\infty}(d^d)^nca^nb=0.$$ Hence by Theorem 3.1, $x^d=0$.\end{proof}

\section{Block operator matrices}

In this section, we shall apply our main results to block operator matrices over a Banach algebra $\mathcal{A}$. We have

\begin{thm} Let $M=\left(
  \begin{array}{cc}
    A & B \\
    C & D
  \end{array}
\right)\in M_2(\mathcal{A})$, $A$ and $D$ have g-Drazin inverses. If $BC=0, AB=\lambda A^{\pi}BD$ and $DC=\lambda D^{\pi}CA$, then $M\in M_2(\mathcal{A})^d$ and
$$M^d=\left(
\begin{array}{cc}
A^d&B(D^d)^2\\
C(A^d)^2&D^d
\end{array}
\right)+\sum\limits_{n=1}^{\infty}M^n\left(
\begin{array}{cc}
0&B(D^d)^{n+2}\\
C(A^d)^{n+2}&0
\end{array}
\right).$$\end{thm}
\begin{proof} Clearly, we have $M=P+Q$, where
$$P=\left(
\begin{array}{cc}
A&0\\
0&D
\end{array}
\right), Q=\left(
\begin{array}{cc}
0&B\\
C&0
\end{array}
\right).$$ Then we have
$$\begin{array}{c}
P^d=\left(
\begin{array}{cc}
A^d&0\\
0&D^d
\end{array}
\right), Q^3=0, Q^d=0,\\
P^{\pi}=\left(
\begin{array}{cc}
A^{\pi}&0\\
0&D^{\pi}
\end{array}
\right), Q^{\pi}=I_2.
\end{array}$$
By the computation, we have
$$\begin{array}{c}
PQ=\left(
\begin{array}{cc}
0&AB\\
DC&0
\end{array}
\right),\\
QP=\left(
\begin{array}{cc}
0&BD\\
CA&0
\end{array}
\right).
\end{array}$$
Hence we have $$P^{\pi}Q^{\pi}QPQ^{\pi}=\left(
\begin{array}{cc}
0&AA^{\pi}BD^{\pi}\\
D^{\pi}DC&0
\end{array}
\right).$$
By hypothesis, we have $PQ=\lambda P^{\pi}Q^{\pi}QPQ^{\pi}$. In light of Theorem 2.4,
we have $$\begin{array}{lll}
M^d&=&P^d+\sum\limits_{n=0}^{\infty}M^nQ(P^d)^{n+2}\\
&=&\left(
\begin{array}{cc}
A^d&0\\
0&D^d
\end{array}
\right)+\sum\limits_{n=0}^{\infty}M^n\left(
\begin{array}{cc}
0&B(D^d)^{n+2}\\
C(A^d)^{n+2}&0
\end{array}
\right),
\end{array}$$ as required.\end{proof}

\begin{cor} Let $M=\left(
  \begin{array}{cc}
    A & B \\
    C & D
  \end{array}
\right)\in M_2(\mathcal{A})$, $A$ and $D$ have g-Drazin inverses. If $CB=0, AB=\lambda A^{\pi}BD$ and $DC=\lambda D^{\pi}CA$, then $M\in M_2(\mathcal{A})^d$ and
$$M^d=\left(
\begin{array}{cc}
A^d&B(D^d)^2\\
C(A^d)^2&D^d
\end{array}
\right)+\sum\limits_{n=1}^{\infty}M^n\left(
\begin{array}{cc}
0&B(D^d)^{n+2}\\
C(A^d)^{n+2}&0
\end{array}
\right).$$\end{cor}
\begin{proof} It is easy to verify that
$$\left(
  \begin{array}{cc}
    A & B \\
    C & D
  \end{array}
\right)=\left(
  \begin{array}{cc}
    0 & I \\
    I & 0
  \end{array}
\right)\left(
  \begin{array}{cc}
    D & C \\
    B & A
  \end{array}
\right)\left(
  \begin{array}{cc}
    0 & I \\
    I & 0
  \end{array}
\right).$$ Applying Theorem 4.1 to the matrix $\left(
  \begin{array}{cc}
    D & C \\
    B & A
  \end{array}
\right)$, we obtain the result.\end{proof}

It is convenient at this stage to include the following theorem.

\begin{thm} Let $M=\left(
  \begin{array}{cc}
    A & B \\
    C & D
  \end{array}
\right)\in M_2(\mathcal{A})$, $A$ and $D$ have g-Drazin inverses. If $BC=0, BD=\lambda A^{\pi}ABD^{\pi}$ and $CA=\lambda D^{\pi}DCA^{\pi}$, then $M\in M_2(\mathcal{A})^d$ and
$$M^d=\left(
\begin{array}{cc}
2A^d&(A^d)^2B\\
(D^d)^2C&2D^d+(D^d)^3CB
\end{array}
\right).$$\end{thm}
\begin{proof} Construct $P$ and $Q$ as in Theorem 4.1, we have
$$Q^{\pi}P^{\pi}PQP^{\pi}=\left(
  \begin{array}{cc}
    0&A^{\pi}ABD^{\pi}\\
    D^{\pi}DCA^{\pi}& D
  \end{array}
\right).$$
By hypothesis, we have
$$QP=\lambda Q^{\pi}P^{\pi}PQP^{\pi}.$$ According to Theorem 2.4,
we have $$\begin{array}{lll}
M^d&=&P^d+\sum\limits_{n=0}^{\infty}(P^d)^{n+1}Q^n\\
&=&\left(
\begin{array}{cc}
2A^d&(A^d)^2B\\
(D^d)^2C&2D^d+(D^d)^3CB
\end{array}
\right),
\end{array}$$ as required.\end{proof}

\begin{cor} Let $M=\left(
  \begin{array}{cc}
    A & B \\
    C & D
  \end{array}
\right)\in M_2(\mathcal{A})$, $A$ and $D$ have g-Drazin inverses. If $CB=0, BD=\lambda A^{\pi}ABD^{\pi}$ and $CA=\lambda D^{\pi}DCA^{\pi}$, then $M\in M_2(\mathcal{A})^d$ and
$$M^d=\left(
\begin{array}{cc}
2A^d+(A^d)^3BC&(A^d)^2B\\
(D^d)^2C&2D^d
\end{array}
\right).$$\end{cor}
\begin{proof} Applying Theorem 4.3, to the matrix $\left(
  \begin{array}{cc}
    D & C \\
    B & A
  \end{array}
\right)$, we complete the proof as in Theorem 4.3.\end{proof}

Finally, we give a numerical example to illustrate Theorem 4.3.

\begin{exam}
Let $M=\left(
  \begin{array}{cc}
    A&B\\
    C&D
  \end{array}
\right)\in M_4({\Bbb C})$, where $$A=\left(
\begin{array}{ccc}
0&1&0\\
0&0&-1\\
0&0&1
\end{array}
\right),
B=\left(
\begin{array}{ccc}
0&2&0\\
0&0&1\\
0&0&0
\end{array}
\right), $$ $$C=\left(
\begin{array}{ccc}
0&1&1\\
0&0&0\\
0&0&0
\end{array}
\right),
D=\left(
\begin{array}{ccc}
0&i&0\\
0&0&i\\
0&0&0
\end{array}
\right).$$
Then $BC=0, BD=2i A^{\pi}ABD^{\pi}$ and $CA=2i D^{\pi}DCA^{\pi}$, and $$M^d=\left(
\begin{array}{cccccc}
0&0&-2&0&0&0\\
0&0&-2&0&0&0\\
0&0&2&0&0&0\\
0&0&0&0&0&0\\
0&0&0&0&0&0\\
0&0&0&0&0&0
\end{array}
\right).$$ In this case, $BD\neq A^{\pi}ABD^{\pi}$.
\end{exam}
\begin{proof} As $D^3=B^3=C^2=0$, these are nilpotent matrices and so $B^d=C^d=D^d=\left(
\begin{array}{ccc}
0&0&0\\
0&0&0\\
0&0&0
\end{array}
\right)$. Then $B^{\pi}=C^{\pi}=D^{\pi}=I_3$. By direct computation, we get $$A^d=\left(
\begin{array}{ccc}
0&0&-1\\
0&0&-1\\
0&0&1
\end{array}
\right), A^{\pi}=\left(
\begin{array}{ccc}
1&0&1\\
0&1&1\\
0&0&0
\end{array}
\right).$$ Also it is easy to see that  $$BD=\left(
\begin{array}{ccc}
0&0&2i\\
0&0&0\\
0&0&0
\end{array}
\right)=2i\left(
\begin{array}{ccc}
0&0&1\\
0&0&0\\
0&0&0
\end{array}
\right)=2i A^{\pi}ABD^{\pi}, $$ $BC=0$, while $BD\neq A^{\pi}ABD^{\pi}$. Moreover, $CA=0=2i D^{\pi}DCA^{\pi}$. Therefore by Theorem 4.3, $$M^d=\left(
\begin{array}{cc}
2A^d&(A^d)^2B\\
(D^d)^2C&2D^d+(D^d)^3CB
\end{array}
\right)=\left(
\begin{array}{cccccc}
0&0&-2&0&0&0\\
0&0&-2&0&0&0\\
0&0&2&0&0&0\\
0&0&0&0&0&0\\
0&0&0&0&0&0\\
0&0&0&0&0&0
\end{array}
\right).$$\end{proof}

\end{document}